\documentclass[11pt]{article}
\usepackage{epigamath}

\setpapertype{A4}

\usepackage[english]{babel}

\usepackage[dvips]{graphicx}     

\makeatletter
\def\@tocline#1#2#3#4#5#6#7{\relax
  \ifnum #1>\c@tocdepth 
  \else
    \par \addpenalty\@secpenalty\addvspace{#2}%
    \begingroup \hyphenpenalty\@M
    \@ifempty{#4}{%
      \@tempdima\csname r@tocindent\number#1\endcsname\relax
    }{%
      \@tempdima#4\relax
    }%
    \parindent\z@ \leftskip#3\relax \advance\leftskip\@tempdima\relax
    \rightskip\@pnumwidth plus4em \parfillskip-\@pnumwidth
    #5\leavevmode\hskip-\@tempdima
      \ifcase #1
       \or\or \hskip 1em \or \hskip 2em \else \hskip 3em \fi%
      #6\nobreak\relax
    \dotfill\hbox to\@pnumwidth{\@tocpagenum{#7}}\par
    \nobreak
    \endgroup
  \fi}
\makeatother

\newcommand{\qedhere}{}

\newtheorem{defn}{Definition}[section]
\newtheorem{ex}[defn]{Example}
\newtheorem{rmk}[defn]{Remark}
\newtheorem{question}[defn]{Question}
\newtheorem{thm}[defn]{Theorem}
\newtheorem{lem}[defn]{Lemma}
\newtheorem{prop}[defn]{Proposition}

\newtheorem{conj}[defn]{Conjecture}

\renewcommand{\qedsymbol}{{$\Box$}}

\def\C{\ensuremath{\mathbb{C}}}

\def\N{\ensuremath{\mathbb{N}}}
\def\P{\ensuremath{\mathbb{P}}}
\def\Q{\ensuremath{\mathbb{Q}}}
\def\R{\ensuremath{\mathbb{R}}}
\def\Z{\ensuremath{\mathbb{Z}}}

\def\FF{\ensuremath{\mathcal F}}

\def\OO{\ensuremath{\mathcal O}}

\def\TT{\ensuremath{\mathcal T}}

\def\ch{\mathop{\mathrm{ch}}\nolimits}

\def\Coh{\mathop{\mathrm{Coh}}\nolimits}

\def\Db{\mathop{\mathrm{D}^{\mathrm{b}}}\nolimits}

\def\dim{\mathop{\mathrm{dim}}\nolimits}

\def\ext{\mathop{\mathrm{ext}}\nolimits}
\def\Ext{\mathop{\mathrm{Ext}}\nolimits}

\def\hom{\mathop{\mathrm{hom}}\nolimits}
\def\Hom{\mathop{\mathrm{Hom}}\nolimits}

\def\RlHom{\mathop{\mathbf{R}\mathcal Hom}\nolimits}

\def\min{\mathop{\mathrm{min}}\nolimits}

\def\Pic{\mathop{\mathrm{Pic}}\nolimits}

\def\td{\mathop{\mathrm{td}}\nolimits}

\def\into{\ensuremath{\hookrightarrow}}
\def\onto{\ensuremath{\twoheadrightarrow}}

\title{Bridgeland Stability Conditions on Fano Threefolds}
\titlemark{Bridgeland Stability Conditions on Fano Threefolds}
\author{Marcello Bernardara, Emanuele Macr\`i, Benjamin Schmidt, and Xiaolei Zhao}
\authoraddresses{
\authordata{M. Bernardara}{\firstname{Marcello} \lastname{Bernardara}\\
\institution{Institut de Math\'ematiques de Toulouse, Universit\'e Paul Sabatier, 118 route de Narbonne, 31062 Toulouse Cedex 9, France}\\
\email{marcello.bernardara@math.univ-toulouse.fr}}

\authordata{E. Macr\`i}{\firstname{Emanuele} \lastname{Macr\`i}\\
\institution{Northeastern University, Department of Mathematics, 360 Huntington Avenue, Boston, MA 02115-5000, USA}\\
\email{e.macri@northeastern.edu}}

\authordata{B. Schmidt}{\firstname{Benjamin} \lastname{Schmidt}\\
\institution{The Ohio State University, Department of Mathematics, 231 W 18th Avenue, Columbus, OH 43210-1174, USA\\
{\it Current address:} Department of Mathematics, The University of Texas at Austin, 2515 Speedway Stop C1200, Austin, TX 78712-1202, USA}\\
\email{schmidt@math.utexas.edu}}

\authordata{X. Zhao}{\firstname{Xiaolei} \lastname{Zhao}\\
\institution{Northeastern University, Department of Mathematics, 360 Huntington Avenue, Boston, MA 02115-5000, USA}\\
\email{x.zhao@northeastern.edu}}
}
\authormark{M. Bernardara, E. Macr\`i, B. Schmidt, \&\ X. Zhao}
\date{\vspace{-5ex}} 
\journal{\'Epijournal de G\'eom\'etrie Alg\'ebrique} 
\acceptation{Received by the Editors on August 21, 2016, and in final form on January 10, 2017. \\ Accepted on March 23, 2017.}

\newcommand{\articlereference}{Volume 1 (2017), Article Nr. 2}

\begin{document}


\maketitle


\begin{prelims}

\def\abstractname{Abstract}
\abstract{We show the existence of Bridgeland stability conditions on all Fano threefolds, by proving a modified version of a conjecture by Bayer, Toda, and the second author.
The key technical ingredient is a strong Bogomolov inequality, proved recently by Chunyi Li.
Additionally, we prove the original conjecture for some toric threefolds by using the toric Frobenius morphism.}



\vspace{0.1cm}

\languagesection{Fran\c{c}ais}{%

\textbf{Titre. Conditions de stabilit\'e de Bridgeland sur les vari\'et\'es de Fano de dimension~3}
\commentskip
\textbf{R\'esum\'e.}
Nous \'etablissons l'existence de conditions de stabilit\'e de Bridgeland sur toute vari\'et\'e de Fano de dimension 3 en montrant une version modifi\'ee d'une conjecture de Bayer, Toda et du deuxi\`eme auteur. L'ingr\'edient technique essentiel est une in\'egalit\'e forte de Bogomolov, d\'emontr\'ee r\'ecemment par Chunyi Li. En outre, nous montrons la conjecture originale pour certaines vari\'et\'es toriques de dimension 3 via l'utilisation du morphisme de Frobenius torique.}

\vspace{2mm}

\end{prelims}


\newpage

\setcounter{tocdepth}{1}
\tableofcontents

\section{Introduction}
\label{sec:intro}

In the recent article \cite{Li15:conjecture_fano_threefold}, Chunyi Li proved the existence of Bridgeland stability conditions on Fano threefolds of Picard number 1. The purpose of this paper is to extend Li's method to any Fano threefold, when we choose the polarization to be the anti-canonical divisor.

Tilt stability was introduced in \cite{BMT14:stability_threefolds} as an intermediate notion to define Bridgeland stability on threefolds.
Its definition mimics Bridgeland's original construction in the surface case \cite{Bri08:stability_k3,AB13:k_trivial}.
Let $X$ be a smooth projective complex threefold, and let $H$ be a polarization on $X$. Also, let $\alpha,\beta\in\R$, $\alpha>0$.
The usual notion of slope stability (with respect to $H$) induces a family of torsion pairs on the category of coherent sheaves $\Coh(X)$, parameterized by $\beta$, roughly obtained by truncating Harder-Narasimhan filtrations at the slope $\beta$.
We denote the corresponding family of tilted abelian categories by $\Coh^\beta(X)$.
Tilt stability is the (weak) stability condition defined on $\Coh^\beta(X)$ with respect to the slope
\[
\nu_{\alpha,\beta} := \frac{H \cdot \ch_2^\beta - \frac{\alpha^2}{2} H^3 \cdot \ch_0^\beta}{H^2 \cdot \ch_1^\beta},
\]
where, as usual, dividing by $0$ equals to $+\infty$.
Here $\ch^\beta$ denotes the twisted Chern character $\ch \cdot e^{-\beta H}$.
We will review all of this in Section \ref{sec:tilt}.

The intuition is that tilt stability on threefolds should play the role of usual slope stability of sheaves on surfaces.
By keeping this idea in mind, we should look for a generalization of the classical Bogomolov inequality for slope stable sheaves on surfaces to tilt stability.
This is exactly the content of our main theorem.
An analogous result, with a similar proof, appears independently in \cite[Theorem 1.3]{Piy16:Fano}.

\begin{thm}\label{thm:main}
Let $X$ be a smooth projective Fano threefold of index $i_X$ and let $H = - \frac{K_X}{i_X}$.
There exists a cycle $\Gamma\in A_1(X)_\R$ such that $\Gamma \cdot H \geq 0$ and, for any $\nu_{\alpha,\beta}$-stable object $E$ with $\nu_{\alpha,\beta}(E) = 0$, we have
\[
\ch_3^\beta(E) - \Gamma \cdot \ch_1^\beta(E) - \frac{\alpha^2}{6} H^2 \cdot \ch_1^{\beta}(E) \leq 0.
\]
\end{thm}

In \cite{BMT14:stability_threefolds} (see also \cite{BMS14:abelian_threefolds}) a version of Theorem \ref{thm:main} was conjectured without the class $\Gamma$. That original conjecture turned out to be imprecise in the case of the blow-up of $\P^3$ at a point \cite{Sch16:counterexample}.
Therefore, modifying the conjecture in \cite{BMT14:stability_threefolds} by adding $\Gamma$ seems like a necessary step to understand Bridgeland stability on threefolds.
In fact, Theorem \ref{thm:main} does still imply existence of Bridgeland stability conditions \cite{Bri07:stability_conditions} on all Fano threefolds exactly as in \cite{BMT14:stability_threefolds}. Indeed, similarly to what was done in \cite{BMS14:abelian_threefolds}, we get the following quadratic inequality (see Proposition \ref{prop:support_property}):
\begin{align*}
Q_{\alpha, \beta}(E) = &\alpha^2\, \left(\overline{\Delta}_H(E) + 3\, \frac{\Gamma\cdot H}{H^3} (H^3\cdot \ch_0^\beta(E))^2\right)\\
& + 2\, (H \cdot \ch_2^{\beta}(E)) (2\, H\cdot \ch_2^\beta(E) - 3\, \Gamma\cdot H \cdot \ch_0^\beta(E)) \\
& - 6 (H^2 \cdot \ch_1^{\beta}(E)) (\ch_3^{\beta}(E) - \Gamma \cdot \ch_1^{\beta}(E)) \geq 0,
\end{align*}
for any $\nu_{\alpha,\beta}$-stable object $E$, where
\[
\overline{\Delta}_H(E) = (H^2 \cdot \ch_1^\beta(E))^2 - 2(H^3 \cdot \ch_0^\beta(E))(H \cdot \ch_2^\beta(E)).
\]

When the Picard rank of $X$ is $1$, Theorem \ref{thm:main} was proved in \cite{Li15:conjecture_fano_threefold} with $\Gamma = 0$ (the cases of $\P^3$ and of the quadric threefold were proved earlier, with a completely different proof, in \cite{Mac14:conjecture_p3} and \cite{Sch14:conjecture_quadric}). The idea of the proof of Theorem \ref{thm:main} follows along the same lines as in \cite{Li15:conjecture_fano_threefold}, by using an Euler characteristic estimate. This will be carried out in Section \ref{sec:main_standard}.
The key result is a slight generalization of \cite[Proposition 3.2]{Li15:conjecture_fano_threefold} to the higher Picard rank case.
This is the content of Theorem \ref{thm:bogomolov_extension}, whose proof will take the whole Section \ref{sec:Bog}.

There are examples in which the class $\Gamma$ is indeed $0$, even if the Picard number is greater than $1$ (see Theorem \ref{thm:toric}). More generally, it would be interesting to know what the optimal class $\Gamma$ is.
For example, all known counter-examples (\cite{Sch16:counterexample, Mar16:counterexample}) to the original conjecture from \cite{BMT14:stability_threefolds} on Fano threefolds do satisfy a stronger inequality for some choice of $\Gamma$ with $\Gamma \cdot H = 0$ (see also Section \ref{sec:examples2}).
It is interesting to observe that when such a strong version of the conjecture holds, then all possible applications of such an inequality
either to the birational geometry of threefolds or to counting invariants (pointed out in a sequence of papers; e.g., \cite{BBMT14:fujita, Tod12:stability_curve_counting, Tod14:icm}) would hold unchanged.
Finally, we remark that the original conjecture is also proved for abelian threefolds (and their \'etale quotients) in \cite{MP15:conjecture_abelian_threefoldsI,MP16:conjecture_abelian_threefoldsII,BMS14:abelian_threefolds}.

In the case of $\P^1 \times \P^2$ (and partly in the case of the blow-up of $\P^3$ in a line) we are able to prove the original conjecture with $\Gamma =0$ by using the toric Frobenius morphism. The proof in these two cases will then follow a similar ``limit'' argument of Euler characteristics as in \cite{BMS14:abelian_threefolds}.

\begin{thm}\label{thm:toric}
Let $X$ be a $\P^2$-bundle over $\P^1$.
Let $H$ be an ample divisor such that, for all effective divisors $D$ on $X$, we have $H \cdot D^2\geq 0$.
Then, for any $\nu_{\alpha,\beta}$-stable object $E\in\Coh^\beta(X)$ with $\nu_{\alpha,\beta}(E)=0$, we have
\[
\ch_3^\beta(E) - \frac{\alpha^2}{6} H^2\cdot\ch_1^\beta(E) \leq 0.
\]
\end{thm}

The assumption on the polarization $H$ is quite strong.
It holds for all polarizations on $\P^1\times\P^2$.
On the blow-up of $\P^3$ in a line, it holds for some polarizations, but not for the anti-canonical bundle.
In fact, we will prove Theorem \ref{thm:toric} under a slightly more general technical assumption. The precise statement is Theorem \ref{thm:toric2}.
In particular, it will hold for all polarizations $H$ on $\P^1\times\P^1\times\P^1$ as well.

\subsection*{Acknowledgments}
We thank Arend Bayer and Paolo Stellari very much for allowing us to include the results in Section \ref{sec:examples1}, which arose from joint discussions with them, and for many discussions and useful advice.
We also thank Cristian Martinez for sending us a copy of his preprint \cite{Mar16:counterexample}, and Zheng Hua, Yukinobu Toda, and Hokuto Uehara for useful discussions. Finally, we would like to thank the referee for carefully reading this paper and bringing to our attention an error in the original
version. The authors would like to acknowledge Northeastern University and the Institut de Math\'ematiques de Toulouse for the hospitality during the writing of this paper.
The second author was partially supported by NSF grant DMS-1523496 and by ANR-11-LABX-0040-CIMI within the program ANR-11-IDEX-0002-02;
and the third author by a Presidential Fellowship of the Ohio State University.

\subsection*{Notation}
\begin{center}
  \begin{tabular}{ r l }
    $X$ & smooth projective threefold over $\C$ \\
    $H$ & fixed ample divisor on $X$ \\
    $\Db(X)$ & bounded derived category of coherent \\ & sheaves on $X$ \\
    $\ch(E)$ & Chern character of an object $E \in \Db(X)$  \\
    $\ch_{\leq l}(E)$ & $(\ch_0(E), \ldots, \ch_l(E))$ \\
    $H \cdot \ch(E)$ & $(H^3 \cdot \ch_0(E), H^2 \cdot \ch_1(E), H \cdot \ch_2(E)
    \ch_3(E))$ \\
    $H \cdot \ch_{\leq l}(E)$ & $(H^3 \cdot
    \ch_0(E), \ldots, H^{3-l} \cdot \ch_l(E))$ \\
    $\ch^{\beta}(E)$ & $e^{-\beta H} \cdot \ch(E)$ \\
    $\ext^i(E,F)$ & $\dim \Ext^i(E,F)$ for $E, F \in \Db(X)$ and $i \in Z$ \\
    $\hom(E,F)$ & $\dim \Hom(E,F)$ for $E, F \in \Db(X)$
  \end{tabular}
\end{center}

\section{Background on tilt stability}
\label{sec:tilt}

In \cite{BMT14:stability_threefolds} the notion of tilt stability was introduced as an auxiliary notion in between slope stability and Bridgeland stability on threefolds. Let $X$ be a smooth projective threefold over the complex numbers and $H$ be a fixed ample divisor on $X$. The classical slope for a coherent sheaf $E \in \Coh(X)$ is defined as
\[\mu(E) := \frac{H^2\cdot \ch_1(E)}{H^3 \cdot \ch_0(E)},\]
where division by zero is interpreted as $+\infty$. As usual a coherent sheaf $E$ is called \emph{slope-(semi)stable} if for any non trivial proper subsheaf $F \subset E$ the inequality $\mu(F) < (\leq) \mu(E/F)$ holds.

Let $\beta$ be an arbitrary real number. Then the twisted Chern character $\ch^{\beta}$ is defined to be $e^{-\beta H} \cdot \ch$. Explicitly:
\begin{align*}
\ch^{\beta}_0 &= \ch_0, \\
\ch^{\beta}_1 &= \ch_1 - \beta H \cdot \ch_0 ,\\
\ch^{\beta}_2 &= \ch_2 - \beta H \cdot \ch_1 + \frac{\beta^2}{2} H^2 \cdot \ch_0, \\
\ch^{\beta}_3 &= \ch_3 - \beta H \cdot \ch_2 + \frac{\beta^2}{2} H^2 \cdot \ch_1 -
\frac{\beta^3}{6} H^3 \cdot \ch_0.
\end{align*}

The process of tilting is used to construct a new heart of a bounded t-structure. For more information on the general theory of tilting we refer to \cite{HRS96:tilting} and \cite{BvdB03:functors}. A torsion pair is defined by
\begin{align*}
\TT_{\beta} &:= \{E \in \Coh(X) : \text{any quotient $E \onto G$ satisfies $\mu(G) > \beta$} \}, \\
\FF_{\beta} &:=  \{E \in \Coh(X) : \text{any subsheaf $F \subset E$ satisfies $\mu(F) \leq \beta$} \}.
\end{align*}
The heart of a bounded t-structure is given as the extension closure $\Coh^{\beta}(X) := \langle \FF_{\beta}[1],\TT_{\beta} \rangle$. Let $\alpha > 0$ be a positive real number. The tilt slope is defined as
\[\nu_{\alpha, \beta} := \frac{H \cdot \ch^{\beta}_2 - \frac{\alpha^2}{2} H^3
\cdot \ch^{\beta}_0}{H^2 \cdot \ch^{\beta}_1}.\]
As before, an object $E \in \Coh^{\beta}(X)$ is called \emph{tilt-(semi)stable} (or \emph{$\nu_{\alpha,\beta}$-(semi)stable}) if for any non trivial proper subobject $F \subset E$ the inequality $\nu_{\alpha, \beta}(F) < (\leq) \nu_{\alpha, \beta}(E/F)$ holds.

\begin{thm}[{Bogomolov Inequality for Tilt Stability, \cite[Corollary 7.3.2]{BMT14:stability_threefolds}}]
\label{thm:bogomolov}
\hfill Any \linebreak $\nu_{\alpha, \beta}$-semistable object $E \in \Coh^{\beta}(X)$ satisfies
\begin{align*}
\overline{\Delta}_H(E) &= (H^2 \cdot \ch_1^{\beta}(E))^2 - 2(H^3 \cdot \ch_0^{\beta}(E))(H \cdot \ch_2^{\beta}(E)) \\
&= (H^2 \cdot \ch_1(E))^2 - 2(H^3 \cdot \ch_0(E))(H \cdot \ch_2(E)) \geq 0.
\end{align*}
\end{thm}

Let $\Lambda = \Z \oplus \Z \oplus \tfrac{1}{2} \Z$ be the image of the map $H \cdot \ch_{\leq 2}$. Notice that $\nu_{\alpha, \beta}$ factors through $H \cdot \ch_{\leq 2}$. Varying $(\alpha, \beta)$ changes the set of stable objects. A \textit{numerical wall} in tilt stability with respect to a class $v \in \Lambda$ is a non trivial proper subset $W$ of the upper half plane given by an equation of the form $\nu_{\alpha, \beta}(v) = \nu_{\alpha, \beta}(w)$ for another class $w \in \Lambda$. A subset $S$ of a numerical wall $W$ is called an \textit{actual wall} if the set of semistable objects with class $v$ changes at $S$. The structure of walls in tilt stability is rather simple. Part (1) - (5) is usually called Bertram's Nested Wall Theorem and appeared first in \cite{Mac14:nested_wall_theorem}, while part (6), (7), and (8) are in Lemma 2.7 and Appendix A of \cite{BMS14:abelian_threefolds}. The last part of (8) about reflexivity is to be found in \cite[Proposition 3.1]{LM16:examples_tilt}.

\begin{thm}[Structure Theorem for Tilt Stability]\label{thm:Bertram}
Let $v \in \Lambda$ be a fixed class. All numerical walls in the following statements are with respect to $v$.
\begin{enumerate}
  \item Numerical walls in tilt stability are either semicircles with center on the $\beta$-axis or rays parallel to the $\alpha$-axis.
  \item If two numerical walls given by classes $w,u \in \Lambda$ intersect, then $v$, $w$ and $u$ are linearly dependent. In particular, the two walls are completely identical.
  \item The curve $\nu_{\alpha, \beta}(v) = 0$ is given by a hyperbola. Moreover, this hyperbola intersects all semicircular walls at their top point.
  \item If $v_0 \neq 0$, there is exactly one numerical vertical wall given by $\beta = v_1/v_0$. If $v_0 = 0$, there is no actual vertical wall.
  \item If a numerical wall has a single point at which it is an actual wall,
  then all of it is an actual wall.
  \item If there is an actual wall numerically defined by an exact sequence of tilt semistable objects $0 \to F \to E \to G \to 0$ such that $H \cdot \ch_{\leq 2}(E) = v$, then
  \[\overline{\Delta}_H(F) + \overline{\Delta}_H(G) \leq \overline{\Delta}_H(E).\]
  Moreover, equality holds if and only if either $H \cdot \ch_{\leq 2}(F) = 0$, $H \cdot \ch_{\leq 2}(G) = 0$, or both $\overline{\Delta}_H(E) = 0$ and $H \cdot \ch_{\leq 2}(F)$, $H \cdot \ch_{\leq 2}(G)$, and $H \cdot \ch_{\leq 2}(E)$ are all proportional.
  \item If $\overline{\Delta}_H(E) = 0$ for a tilt semistable object $E$, then $E$ can only be destabilized at the unique numerical vertical wall.
  \item If $E$ is a tilt stable object for fixed $\beta \in \R$ and all $\alpha \gg 0$, then $E$ is one of the following.
  \begin{itemize}
  \item If $\ch_0(E) \geq 0$, then $E$ is a slope semistable sheaf.
  \item If $\ch_0(E) < 0$, then $H^0(E)$ is a torsion sheaf supported in dimension smaller than or equal to $1$ and $H^{-1}(E)$ is a reflexive slope semistable sheaf with positive rank.
  \end{itemize}
\end{enumerate}
\end{thm}

A generalized Bogomolov type inequality involving third Chern characters for tilt semistable objects with $\nu_{\alpha, \beta} = 0$ has been conjectured in \cite{BMT14:stability_threefolds}. Its main goal was the construction of Bridgeland stability conditions on arbitrary threefolds.

\begin{conj}[{\cite[Conjecture 1.3.1]{BMT14:stability_threefolds}}]
\label{conj:bmt}
For any $\nu_{\alpha,\beta}$-stable object $E \in \Coh^{\beta}(X)$ with $\nu_{\alpha,\beta}(E)=0$ the inequality
\[
\ch^{\beta}_3(E) \leq \frac{\alpha^2}{6} H^2 \cdot \ch^{\beta}_1(E)
\]
holds.
\end{conj}

The conjecture has been proved for $\P^3$ in \cite{Mac14:conjecture_p3} and for the smooth quadric hypersurface $Q \subset \P^4$ in \cite{Sch14:conjecture_quadric}. All other Fano threefolds of Picard rank one were handled in \cite{Li15:conjecture_fano_threefold}. Finally, it is known to hold for abelian threefolds with independent proofs in \cite{BMS14:abelian_threefolds} and \cite{MP15:conjecture_abelian_threefoldsI, MP16:conjecture_abelian_threefoldsII}. It turns out that the conjecture fails on the blow up of $\P^3$ in a single point as shown in \cite{Sch16:counterexample}. In this article we give an affirmative answer to the following natural follow up question in case $X$ is a Fano threefold and the polarization $H$ is given by the anticanonical divisor.

\begin{question}
\label{q:new_conjecture}
{\rm Is there a cycle $\Gamma\in A_1(X)_\R$ depending at most on $H$ such that $\Gamma \cdot H \geq 0$ and for any $\nu_{\alpha,\beta}$-stable object $E$ with $\nu_{\alpha,\beta}(E)=0$, we have
\[
\ch_3^\beta(E) \leq \Gamma \cdot \ch_1^\beta(E) + \frac{\alpha^2}{6} H^2 \cdot \ch_1^{\beta}(E)?
\]}
\end{question}

The condition $\Gamma \cdot H \geq 0$ is crucial for the reduction to the case $\alpha = 0$. In order to state it precisely, we first need the notion of $\overline{\beta}$-stability.

\begin{defn}{\rm
For any object $E \in \Coh^{\beta}(X)$, we define
\[
\overline{\beta}(E) = \begin{cases}
\frac{H^2 \cdot \ch_1(E) - \sqrt{\overline{\Delta}_H(E)}}{H^3 \cdot \ch_0(E)} & \ch_0(E) \neq 0, \\
\frac{H \cdot \ch_2(E)}{H^2 \cdot \ch_1(E)} & \ch_0(E) = 0.
\end{cases}
\]
Moreover, we say that $E$ is \textit{$\overline{\beta}$-(semi)stable}, if it is (semi)stable in a neighborhood of $(0, \overline{\beta}(E))$.}\end{defn}

By this definition we have $H \cdot\ch_2^{\overline{\beta}(E)}(E) = 0$.

\begin{prop}[{\cite[Proposition 5.1.3]{BMT14:stability_threefolds}}]
\label{prop:tilt_derived_dual}
Assume $E \in \Coh^{\beta}(X)$ is $\nu_{\alpha, \beta}$-semistable with $\nu_{\alpha, \beta}(E) \neq \infty$. Then there is a $\nu_{\alpha, -\beta}$-semistable object $\tilde{E} \in \Coh^{-\beta}(X)$ and a sheaf $T$ supported in dimension $0$ together with a triangle
\[\tilde{E} \to \RlHom(E, \OO_X)[1] \to T[-1] \to \tilde{E}[1].\]
\end{prop}

The following lemma was first proved for the original conjecture in \cite{BMS14:abelian_threefolds}, but it works out in our case as well without any change in the proof. 

\begin{lem}
\label{lem:reduction_alpha_0}
Let $\Gamma \in A_1(X)_\R$ be a cycle such that $\Gamma \cdot H \geq 0$. Assume that for any $\overline{\beta}$-stable object $E \in \Coh^{\beta}(X)$ with $\overline{\beta}(E) \in [0,1)$ and $\ch_0(E) \geq 0$ the inequality
\[
\ch_3^{\overline{\beta}}(E) \leq \Gamma \cdot \ch_1^{\overline{\beta}}(E)
\]
holds. Then Question \ref{q:new_conjecture} has an affirmative answer with this cycle $\Gamma$.
\end{lem}

\begin{proof}
By Proposition \ref{prop:tilt_derived_dual} we can use the derived dual to reduce to the case $\ch_0(E) \geq 0$. Tensoring with lines bundles $\OO(aH)$ for $a \in \Z$ makes it possible to further reduce to $\overline{\beta}(E) \in [0,1)$.

A straightforward computation shows that
\[
\Gamma \cdot \ch_1^\beta(E) + \frac{\alpha^2}{6} H^2 \cdot \ch_1^{\beta}(E) - \ch_3^\beta(E)
\]
is decreasing along the hyperbola $\nu_{\alpha,\beta}(E)=0$ as $\alpha$ decreases, because $\Gamma \cdot H \geq 0$ and $\ch_0(E) \geq 0$. By using Theorem \ref{thm:Bertram}, (6), we can then proceed by induction on $\overline{\Delta}_H(E)$ to show that it is enough to prove the inequality for $\overline{\beta}$-stable objects.

Indeed, if $\overline{\Delta}_H(E) = 0$, then $E$ is $\overline{\beta}$-stable implying the claim.
Assume $\overline{\Delta}_H(E) > 0$. If $E$ is $\overline{\beta}$-stable, we are done. Otherwise, $E$ is destabilized along a wall between $(\alpha, \beta)$ and $(0, \overline{\beta}(E))$. Let $F_1, \ldots, F_n$ be the stable factors of $E$ along this wall. By induction, the inequality holds for $F_1, \ldots, F_n$ and so it does for $E$ by linearity of the Chern character.
\hfill $\Box$
\end{proof}

As in \cite{BMS14:abelian_threefolds}, we get a quadratic inequality for any tilt semistable object. This still implies the support property for Bridgeland stability conditions, as in \cite[Section 8]{BMS14:abelian_threefolds}.

\begin{prop}\label{prop:support_property}
Assume that Question 2.4 has an affirmative answer.
Then any $\nu_{\alpha, \beta}$-stable object $E$ satisfies
\begin{align*}
Q_{\alpha, \beta}(E) = &\alpha^2\, \left(\overline{\Delta}_H(E) + 3\, \frac{\Gamma\cdot H}{H^3} (H^3\cdot \ch_0^\beta(E))^2\right)\\
& + 2\, (H \cdot \ch_2^{\beta}(E)) (2\, H\cdot \ch_2^\beta(E) - 3\, \Gamma\cdot H \cdot \ch_0^\beta(E)) \\
& - 6 (H^2 \cdot \ch_1^{\beta}(E)) (\ch_3^{\beta}(E) - \Gamma \cdot \ch_1^{\beta}(E)) \geq 0.
\end{align*}
\end{prop}

\begin{proof}
Observe that if $\nu_{\alpha, \beta}(E) = 0$, then $Q_{\alpha, \beta}(E) \geq 0$ is equivalent to $\ch_3^{\beta}(E) \leq \Gamma \cdot \ch_1^{\beta}(E) + \tfrac{\alpha^2}{6} H^2 \cdot \ch_1^{\beta}(E)$. Let $E \in \Coh^{\beta}(X)$ be a $\nu_{\alpha, \beta}$-stable object with $\ch(E) = v$.
If $(\alpha, \beta)$ is on the unique numerical vertical wall for $v$, then $H^2\cdot\ch_1^{\beta}(E) = 0$, $H\cdot \ch_2^\beta(E)\geq0$, and $\ch_0^\beta(E)\leq0$. Therefore, since $\Gamma\cdot H\geq0$, we have
$Q_{\alpha, \beta}(E) \geq 0$.

As a consequence, $(\alpha, \beta)$ lies on a unique numerical semicircular wall $W$ with respect to $v$. One computes that there are $x,y \in \R$, $\delta\in\R_{\geq0}$ such that
\[
Q_{\alpha, \beta}(E) \geq 0 \Leftrightarrow \delta \alpha^2 + \delta \beta^2 + x \beta + y \alpha \geq 0.
\]
Moreover, the equality $Q_{\alpha, \beta}(E) = 0$ holds if and only if
\[
\nu_{\alpha, \beta}(E) = \nu_{\alpha, \beta}\left(H^2 \cdot \ch_1(E), 2H \cdot \ch_2(E)- 3H \cdot \Gamma\cdot\ch_0(E), 3\ch_3(E) - 3\Gamma \cdot \ch_1(E), 0 \right).
\]
In particular, the equation $Q_{\alpha, \beta}(E) \geq 0$ defines either the complement of a semi-disc with center on the
$\beta$-axis or a quadrant to one side of a vertical line. Moreover, $Q_{\alpha, \beta}(E) = 0$ is a numerical wall
with respect to $v$. Since numerical walls do not intersect, the inequality holds on either any or no point of $W$. Computing it at the top point of $W$ concludes the proof.
\hfill $\Box$
\end{proof}

\section{Li's Bogomolov inequality}\label{sec:Bog}

Let $X$ be a Fano threefold of index $i_X$. Consider tilt stability with respect to the polarization $H=-\frac{K_X}{i_X}$. For any tilt semistable object $E$ with $\ch_0(E) \neq 0$, we define

\begin{align*}
\beta_{-}(E) &= \mu(E) - \sqrt{\frac{\overline{\Delta}_H(E)}{(H^3 \cdot \ch_0(E))^2}} \\
\beta_{+}(E) &= \mu(E) + \sqrt{\frac{\overline{\Delta}_H(E)}{(H^3 \cdot \ch_0(E))^2}}.
\end{align*}

If $\ch_0(E) > 0$, we have $\beta_{-}(E) = \overline{\beta}$. If $\ch_0(E) < 0$, then $\beta_{+}(E) = \overline{\beta}$. Note that $\beta_{-}(E) \leq \beta_{+}(E)$ are the two solutions to the equation $\nu_{0, \beta}(E) = 0$. The main result of this section is a slight modification of \cite[Proposition 3.2]{Li15:conjecture_fano_threefold} beyond Picard rank one.

\begin{thm}
\label{thm:bogomolov_extension}
Let $E$ be a tilt stable object other than a shift of $\OO_X$ or an ideal sheaf of points. If additionally $\ch_0(E) \neq 0$ and $0 \leq \beta_{-}(E) \leq \beta_{+}(E) < 1$, then the inequality
\[
\frac{\overline{\Delta}_H(E)}{(H^3 \cdot \ch_0(E))^2} \geq \min \left\{\frac{1}{(H^3)^2}, \frac{3}{2 i_X H^3} \right\}
\]
holds. 
\end{thm}

\begin{rmk}
{\rm Note, that when $X$ has index $1$ or $2$ and its Picard rank is at least $2$, then $H^3 \geq 2$, and
\[
\min \left\{\frac{1}{(H^3)^2}, \frac{3}{2 i_X H^3} \right\} = \frac{1}{(H^3)^2}.
\]
That turns out to be the only case in which we use this theorem in the following sections.}
\end{rmk}

The idea of the proof is as follows. Assuming there exists an object $E$ contradicting the theorem, we will show that $E$ can be chosen such that either $E$ or $E[1]$ is tilt stable for all $\alpha > 0$ and $\beta \in \R$. The conditions on $\beta_{\pm}(E)$ will then allow to prove $\ext^2(E,E) = 0$. A contradiction can be obtained from estimating the Euler characteristic $\chi(E,E) \leq 1$. We will fill the details in a series of lemmas.

\begin{lem}
\label{lem:li_rank_one}
If $E$ is a tilt stable object with rank $\pm 1$, $0 \leq \beta_{-}(E) \leq \beta_{+}(E) < 1$, and $\overline{\Delta}_H(E) = 0$, then $E$ is a shift of $\OO_X$ or an ideal sheaf of points. That means $\overline{\beta}(E) = 0$. In particular, Theorem \ref{thm:bogomolov_extension} holds for rank $\pm 1$ objects.
\end{lem}
\begin{proof}
The Hodge Index Theorem implies
\[
H \cdot (\ch_1(E)^2 - 2 \ch_0(E) \ch_2(E)) \leq \frac{\overline{\Delta}_H(E)}{H^3} = 0
\]
with equality if and only if $\ch_1(E)$ is numerically equivalent to a multiple of $H$. Part (7) of the Structure Theorem for Tilt Stability shows that $E$ is stable for $\alpha \gg 0$, and therefore, part (8) applies. In any of the two cases, $E$ satisfies the classical Bogomolov inequality and we get $H \cdot (\ch_1(E)^2 - 2 \ch_2(E)) = 0$. In particular, $\ch_1(E)$ is numerically equivalent to a multiple of $H$.

If $\ch_0(E) = -1$, then $H^{-1}(E)$ is reflexive of rank one, i.e., a line bundle. Since line bundles have no extensions with skyscraper sheaves, we must have $H^0(E) = 0$. The hypotheses on $\beta_{\pm}(E)$ directly imply $E \cong \OO_X[1]$.

Assume $\ch_0(E) = 1$. Then $E \otimes \OO(-\ch_1(E))$ is an ideal sheaf of a subscheme of dimension smaller than or equal to one. Its second Chern character equals
\[
\frac{\ch_1(E)^2}{2} - \ch_2(E).
\]
Since $H$ is ample and it intersects this curve class as zero, we must have that $E \otimes \OO(-\ch_1(E))$ is an ideal sheaf of a finite number of points.
Again, the hypotheses on $\beta_{\pm}(E)$ directly imply the claim.
\hfill $\Box$
\end{proof}

\begin{lem}
\label{lem:derivative_hyperbola}
Let $E \in \Db(X)$ such that $\ch_0(E) \neq 0$. Then the derivative of $\alpha$ by $\beta$ along the hyperbola $\nu_{\alpha, \beta}(E) = 0$ is given by
\[
\frac{d\alpha}{d\beta} = \frac{\beta - \mu(E)}{\alpha}.
\]
\end{lem}

\begin{proof}
Basic calculus shows
\begin{align*}
0 &= \frac{d}{d \beta} \left( \frac{H^3 \cdot \ch_0(E)}{2} \alpha^2 - \frac{H^3 \cdot \ch_0(E)}{2} \beta^2 + H^2 \cdot \ch_1(E) \beta - H \cdot \ch_2(E) \right) \\
&= H^3 \cdot \ch_0(E) \alpha \frac{d\alpha}{d\beta} -  H^3 \cdot \ch_0(E) \beta + H^2 \cdot \ch_1(E). \qedhere
\end{align*}
\vspace{-1.25cm}

\hfill $\Box$

\vspace{0.1cm}
\end{proof}

\begin{lem}
\label{lem:slope_inequality}
Let $E$ be a tilt semistable object with $\ch_0(E) > 0$. Then there is a stable factor $F$ in the Jordan-H\"older filtration of $E$ satisfying $\mu(F) \leq \mu(E)$ and $\ch_0(F) > 0$. If instead $\ch_0(E) < 0$, then we get $\mu(F) \geq \mu(E)$ and $\ch_0(F) < 0$.

In particular, let $E$ be a tilt stable object with $\ch_0(E) > 0$ that destabilizes at a semicircular wall $W$. Then there is a stable factor $F$ in the Jordan-H\"older filtration of $E$ at $W$ satisfying $\mu(F) < \mu(E)$ and $\ch_0(F) > 0$. If instead $\ch_0(E) < 0$, then we get $\mu(F) > \mu(E)$ and $\ch_0(F) < 0$.
\end{lem}
\begin{proof}
We will only do the case $\ch_0(E) > 0$, and leave $\ch_0(E) < 0$ to the reader. The proof is by induction on the number of stable factors $n$ in a Jordan-H\"older filtration of $E$. If $n = 1$, then we can simply choose $E = F$.

Assume the statement holds for some $n \in \N$ and that $E$ has a Jordan-H\"older filtration of length $n+1$. Let $F \into E$ be a stable subobject in such a Jordan-H\"older filtration of $E$, and let $G$ be the quotient $F/E$. Because of $\ch_0(E) > 0$, we can divide the argument into the following three cases.

\begin{itemize}
\item Assume that $\ch_0(F) \geq \ch_0(E) > 0$ holds. That means we have
\[
0 \leq H^2 \cdot \ch_1(F) - \beta H^3 \cdot \ch_0(F) \leq H^2 \cdot \ch_1(E) - \beta H^3 \cdot \ch_0(E).
\]
This implies
\[
\mu(F) - \beta \leq \frac{H^2 \cdot \ch_1(E) - \beta H^3 \cdot \ch_0(E)}{H^3 \cdot \ch_0(F)} \leq \mu(E) - \beta.
\]
\item If $\ch_0(G) \geq \ch_0(E) > 0$ holds, then we get $\mu(G) \leq \mu(E)$, but $G$ is only semistable. By the inductive hypothesis $G$ has a stable factor with the desired properties.
\item Assume that both $\ch_0(E) > \ch_0(F) > 0$ and $\ch_0(E) > \ch_0(G) > 0$. The vectors $v_1 = (-H^2 \cdot \ch_1(F), H^3 \cdot \ch_0(F))$ and $v_2 =  (-H^2 \cdot \ch_1(G), H^3 \cdot \ch_0(G))$ lie in the upper half plane and add up to $v = (-H^2 \cdot \ch_1(E), H^3 \cdot \ch_0(E))$. It follows that exactly one of the slopes $\mu(F)$ and $\mu(G)$ is larger than or equal to $\mu(E)$. If $\mu(F) \leq \mu(E)$, we have proved the statement. If $\mu(G) \leq \mu(E)$, then we can use the inductive hypothesis on $G$ to obtain an object with the desired properties.
\end{itemize}

For the last part observe that since the wall is not vertical, we get $\mu(F) \neq \mu(E)$ for all stable factors $F$.
\hfill $\Box$
\end{proof}

\begin{lem}
\label{lem:stable_in_quadrant}
Assume there is an object $E$ contradicting Theorem \ref{thm:bogomolov_extension} such that $\overline{\Delta}_H(E) \geq 0$ is minimal among all such objects. Then $E$ can only be destabilized at the unique numerical vertical wall.
\begin{proof}
By Lemma \ref{lem:li_rank_one}, such an object $E$ must satisfy $|\ch_0(E)| \geq 2$. We will give the proof in the case $\ch_0(E) \geq 2$. The negative rank case is almost identical. \footnote{Mirror Figure \ref{fig:reduction_to_minimal} at the line $\beta = \mu(E)$ in that case.}

Assume $E$ is destabilized at some semicircular wall $W$. We will show that there is a stable factor $F$ of $E$ at $W$ that also contradicts Theorem \ref{thm:bogomolov_extension}. Then, part (6) of the Structure Theorem for Walls in Tilt Stability (Theorem \ref{thm:Bertram}) creates a contradiction to the minimality of $\overline{\Delta}_H(E)$.

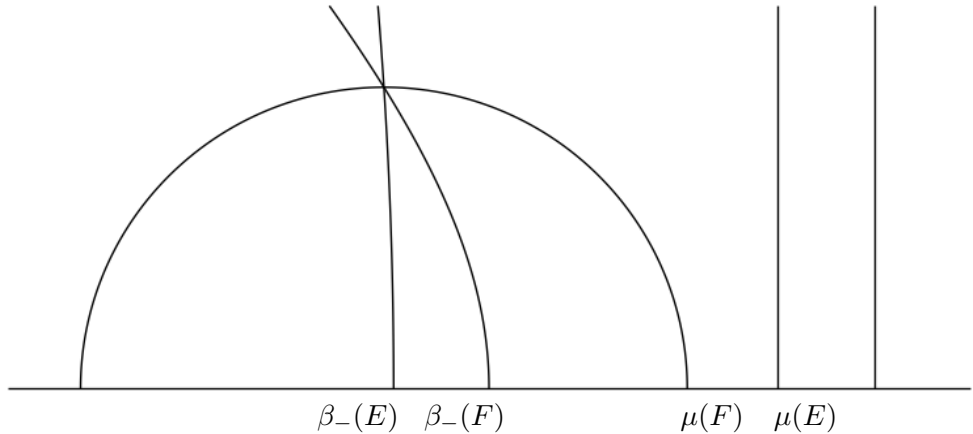
\begin{figure}[ht]
        \centering
        \begin{picture}(140,155)(0,0)
          \put(-120,0){\includegraphics[width=0.8\textwidth]{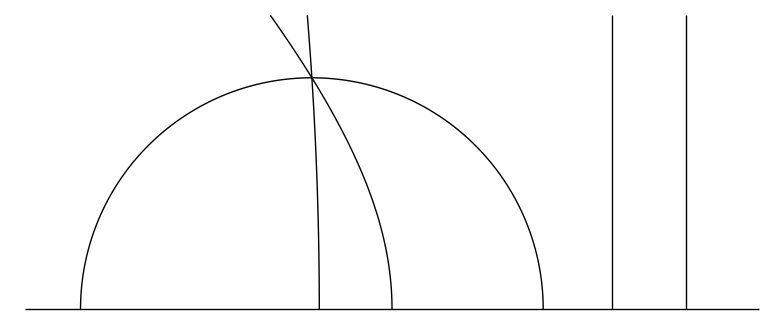}}
          \put(179, -6) {$\mu(E)$}
          \put(144, -6) {$\mu(F)$}
          \put(48, -6) {$\beta_{-}(F)$}
          \put(8, -6) {$\beta_{-}(E)$}
        \end{picture}
        \caption{Visualization of $\beta_{-}(E) < \beta_{-}(F) \leq \mu(F) < \mu(E)$. (Figure created with \cite{sage}.)}
        \label{fig:reduction_to_minimal}
\end{figure}

By Lemma \ref{lem:slope_inequality} there is a stable factor $F$ of $E$ such that $\mu(F) < \mu(E)$ and $\ch_0(F) > 0$. Since the ranks of both $E$ and $F$ are positive, the wall $W$ is to the left of the vertical numerical walls of both $E$ and $F$. In particular, the hyperbolas $\nu_{\alpha, \beta}(E) = 0$ and $\nu_{\alpha, \beta}(F) = 0$ are both decreasing. Let $(\alpha_0, \beta_0)$ be the top point on the semicircle $W$. By Lemma \ref{lem:derivative_hyperbola} the derivative of $\alpha$ by $\beta$ at $(\alpha_0, \beta_0)$ along the hyperbola $\nu_{\alpha, \beta}(E) = 0$ is smaller than along the hyperbola $\nu_{\alpha, \beta}(F) = 0$ (see Figure \ref{fig:reduction_to_minimal}). This proves $\beta_{-}(E) < \beta_{-}(F) \leq \mu(F) < \mu(E)$. Therefore,
\[
\frac{\overline{\Delta}_H(F)}{(H^3 \cdot \ch_0(F))^2} < \frac{\overline{\Delta}_H(E)}{(H^3 \cdot \ch_0(E))^2}
\]
holds. From here we see that $F$ also contradicts the theorem.
\hfill $\Box$
\end{proof}
\end{lem}

\begin{lem}
\label{lem:stable_everywhere_beta_0}
Assume that $E$ is a tilt stable object for some $\alpha > 0$ and $\beta \in \R$ with $\overline{\Delta}_H(E) = 0$, $\overline{\beta}(E) = 0$, and $|\ch_0(E)| \geq 2$. Then there is an object $E'$ with
$\overline{\Delta}_H(E') = 0$, $\overline{\beta}(E') = 0$, and $\ch_0(E') \geq 2$ such that for all $\alpha > 0$ and $\beta \in \R$ either $E'$ or $E'[1]$ is tilt stable.
\end{lem}
\begin{proof}
By hypothesis we know $H \cdot \ch_{\leq 2}(E) = (r,0,0)$, for $r = \ch_0(E)$. Theorem \ref{prop:tilt_derived_dual} implies that we can use the derived dual to reduce to $r \geq 2$. By Theorem \ref{thm:Bertram} (7), $E$ can only destabilize at the vertical wall $\beta = 0$. Moreover, by part (8) of the same Theorem, $E$ is a torsion free slope semistable sheaf.

Let $E[1] \onto G$ be a stable quotient in a Jordan-H\"older filtration of $E[1]$ at the numerical vertical wall. By Theorem \ref{thm:Bertram} (6), we have $H \cdot \ch_{\leq 2}(G) =(r',0,0)$, for $r' = \ch_0(G) \leq 0$.
Since $E$ is a sheaf, $G$ cannot be a torsion sheaf, i.e., $r' \neq 0$.
If $r' \leq -2$, then we can choose $E' = G[-1]$.
If $r' = -1$, then by Lemma \ref{lem:li_rank_one}, $G \cong \OO_X[1]$.
Consider the map $E\to \OO_X$.
Since $\overline{\Delta}_H(E)=0$, we must have that both kernel and image are also slope semistable sheaves with $\overline{\Delta}_H=0$ and slope zero. In particular, they all have the same tilt slope independent of $\alpha$ and $\beta$. Since $r\geq2$, both have also non-zero ranks, and therefore, $E$ is strictly tilt semistable, a contradiction.
\hfill $\Box$
\end{proof}

\begin{lem}
\label{lem:stable_everywhere}
Assume there is an object $E$ contradicting Theorem \ref{thm:bogomolov_extension}. Then $E$ can be chosen such that for all $\alpha > 0$ and $\beta \in \R$ either $E$ or $E[1]$ is tilt stable and $\ch_0(E) \geq 2$.
\end{lem}
\begin{proof}
Assume $E'$ contradicts the theorem such that $\overline{\Delta}_H(E') \geq 0$ is minimal among all such objects. By Lemma \ref{lem:li_rank_one} the inequality $|\ch_0(E)| \geq 2$ is automatic. If $E'$ or $E'[1]$ is tilt stable for all $(\alpha, \beta)$, we can set $E = E'$ and be done. If $E'$ or $E'[-1]$ is tilt stable for all $(\alpha, \beta)$, we can set $E = E'[-1]$ and be done. Otherwise, we can use Lemma \ref{lem:stable_in_quadrant} to show that $E'$ destabilizes at the vertical wall. Let $F_1, \ldots, F_n$ be all the stable factors of $E'$ that satisfy $H \cdot \ch_{\leq 2}(F_i) \neq 0$ for $i = 1, \ldots, n$.

The shape of the wall implies either $\beta = \mu(F_i) = \mu(E')$ or $(H^3 \cdot \ch_0(F_i), H^2 \cdot \ch_1(F_i)) = (0,0)$. By definition of $\Coh^{\beta}(X)$ the rank of an object at the vertical wall can never be positive, and rank $0$ objects have to be torsion sheaves. Therefore, the vectors
$v_i = (-H \cdot \ch_2^{\beta}(F_i), -H^3 \cdot \ch_0^\beta(F_i))$ lie either in the upper half plane or on the negative real line. They also add up to $v = (-H \cdot \ch_2^{\beta}(E'), -H^3 \cdot \ch_0^\beta(E'))$. This is only possible if at least one of the slopes
\[
-\frac{H \cdot \ch_2^{\beta}(F_i)}{H^3 \cdot \ch_0^\beta(F_i)} = \frac{\overline{\Delta}_H(F_i)}{(H^3 \cdot \ch_0(F_i))^2}
\]
is smaller than or equal to the slope
\[
-\frac{H \cdot \ch_2^{\beta}(E')}{H^3 \cdot \ch_0^\beta(E')} = \frac{\overline{\Delta}_H(E')}{(H^3 \cdot \ch_0(E'))^2}.
\]
If $\mu(E) = 0$, then the condition $\beta_{-}(E) \geq 0$ implies $\overline{\Delta}_H(E) = 0$ and $\overline{\beta}(E) = 0$, and we are done by Lemma \ref{lem:stable_everywhere_beta_0}. In all other cases, $F_i$ cannot be $\OO_X[1]$. By part (6) of Theorem \ref{thm:Bertram}, $\overline{\Delta}_H(F_i)$ is minimal again. In particular, the object $E = F_i[-1]$ also contradicts the theorem.
\hfill $\Box$
\end{proof}

\begin{proof}[Proof of Theorem \ref{thm:bogomolov_extension}]
Assume the theorem does not hold, and choose $E$ contradicting it as in Lemma \ref{lem:stable_everywhere}. Then $\ch_0(E) \geq 2$. Stability of $E$ implies $\hom(E,E) = 1$. By assumption the inequality
\[
\overline{\beta}(E(-i_X H)[1]) = \beta_{+}(E) - i_X < \beta_{-}(E) = \overline{\beta}(E)
\]
holds. Due to the fact that $E(-i_X H)[1]$ and $E$ are both stable at $\overline{\beta}(E)$, we get $\ext^2(E,E) = \hom(E, E(-i_X H)[1]) = 0$. The Hirzebruch-Riemann-Roch Theorem together with the Hodge-Index Theorem implies
\begin{align*}
1 &\geq \chi(E,E) = \ch_0(E)^2 - \frac{i_X H \cdot (\ch_1(E)^2 - 2 \ch_0(E) \cdot \ch_2(E))}{2} \\
 &\geq  \ch_0(E)^2 - \frac{i_X \overline{\Delta}_H(E)}{2 H^3}.
\end{align*}
By rearranging the terms one gets the contradiction
\[
\frac{\overline{\Delta}_H(E)}{(H^3 \cdot \ch_0(E))^2} \geq \frac{2(\ch_0(E)^2 -1)}{i_X H^3 \cdot \ch_0(E)^2} \geq \frac{3}{2 i_X H^3}. \qedhere
\]
\vspace{-1.34cm}

\hfill $\Box$

\vspace{0.12cm}
\end{proof}

\section{The main theorem}
\label{sec:main_standard}

Let $X$ be a Fano threefold of index $i_X$ with polarization $H=-\frac{K_X}{i_X}$ and consider tilt-stability with respect to it. In this section we prove Theorem \ref{thm:main}. Chunyi Li already showed that the theorem holds with $\Gamma = 0$ for Picard rank $1$ in \cite{Li15:conjecture_fano_threefold}. Therefore, we only have to deal with the case where the Picard rank is at least two.

The idea of the proof is similar to Li's approach.
As reviewed in Lemma \ref{lem:reduction_alpha_0}, we will assume we have an object $E$ which is $\overline{\beta}$-stable with $0\leq \overline{\beta}(E) < 1$ and $\ch_0(E)\geq0$. Then stability is used to bound the Euler characteristic $\chi(E(-H))$. Finally, Theorem \ref{thm:bogomolov_extension} and the Hirzebruch-Riemann-Roch Theorem allow us to deduce the bound on the Chern character.
To simplify notation, we will often write $\overline{\beta}(E)=\overline{\beta}$.

\subsection{The index two case}
\label{subsec:index2}

First of all, recall that there are only three Fano threefolds with index two and large Picard number.

\begin{thm}[{\cite[Theorem 1]{Fuj82:delta_genus}}]
A Fano threefold $X$ of index two and Picard rank greater than or equal to two is given by either
\begin{enumerate}
\item the blow up of $\P^3$ in a point, or
\item $\P^1 \times \P^1 \times \P^1$, or
\item the projective bundle $\P(T_{\P^2})$, where $T_{\P^2}$ is the tangent bundle of $\P^2$.
\end{enumerate}
\end{thm}

In this section, we will prove Theorem \ref{thm:main} in the latter two cases.
For $\P^1\times\P^1\times\P^1$ we will actually prove the statement for all polarizations in Section \ref{sec:examples1} with a different method. The following lemma holds for all three cases.

\begin{lem}
\label{lem:vanishing2}
Let $E$ be a $\overline{\beta}$-stable object with $0 \leq \overline{\beta}(E) < 1$. Then
\[
\Hom(\OO(H),E) = \Ext^2(\OO(H),E) = 0.
\]
In particular,
\begin{align*}
0 \geq \chi(\OO(H), E) &= \ch_3(E) - \frac{\ch_2(T_X)}{12} \cdot \ch_1(E) + \frac{H \cdot \ch_2(T_X) - 2H^3 + 12}{12} \ch_0(E) \\
&= \ch_3^{\overline{\beta}}(E) - \frac{\ch_2(T_X)}{12} \cdot \ch_1^{\overline{\beta}}(E) + \frac{\overline{\beta}^2}{2} H^2 \cdot \ch_1^{\overline{\beta}}(E) \\
& \ \ \ + \frac{2\overline{\beta}^3 H^3 + (1 - \overline{\beta}) H \cdot \ch_2(T_X) - 2 H^3 + 12}{12} \ch_0^{\overline{\beta}}(E).
\end{align*}
\begin{proof}
Since $\overline{\beta}(E) < 1$, we get $\Hom(\OO(H), E) = 0$ and $\overline{\beta}(E) \geq 0$ shows $\Hom(E, \OO(-H)[1]) = 0$. Together with Serre duality, we have
\[
\chi(O(H), E) \leq \hom(\OO(H), E) + \ext^2(\OO(H), E) = 0.
\]
Since $\chi(\OO_X, \OO_X) = 1$, together with the Hirzebruch-Riemann-Roch Theorem, we get
\[
\td(T_X) = (1, H, \td_2(T_X), 1) = \left(1, H, -\frac{\ch_2(T_X)}{12} + \frac{H^2}{2},1 \right).
\]
Another application of the Hirzebruch-Riemann-Roch Theorem leads to
\[
\chi(O(H), E) = \ch_3(E) - \frac{\ch_2(T_X)}{12} \cdot \ch_1(E) + \frac{H \cdot \ch_2(T_X) - 2H^3 + 12}{12} \ch_0(E).
\]
Finally, the last equality in the lemma is a straightforward computation using $H \cdot \ch_2^{\overline{\beta}}(E) = 0$.
\hfill $\Box$
\end{proof}
\end{lem}

Throughout the rest of this section we will assume that $X$ is either $\P^1 \times \P^1 \times \P^1$ or $\P(T_{\P^2})$. In fact, we will choose $\Gamma = 0$ and prove the original Conjecture \ref{conj:bmt}. This does not work for the blow up of $\P^3$ in a point due to the counterexample in \cite{Sch16:counterexample}. We will handle this case later individually.

\begin{lem}
\label{lem:ch2_index2}
Let $X$ be either $\P^1 \times \P^1 \times \P^1$ or $\P(T_{\P^2})$. The equalities $\ch_2(T_X) = 0$ and $H^3 = 6$ hold.
\end{lem}

\begin{proof}
These are standard calculations based on Chow-K\"unneth formulas for the first case and the
Euler sequence on $\P^2$ for the second case.
\hfill $\Box$
\end{proof}

We are now in position to prove Theorem \ref{thm:main} for both $\P^1 \times \P^1 \times \P^1$ and $\P(T_{\P^2})$ with $\Gamma = 0$. Assume we have an object $E$ which is $\overline{\beta}$-stable with $0 \leq \overline{\beta}(E) < 1$ and $\ch_0(E) \geq 0$. Then Lemma \ref{lem:vanishing2} and Lemma \ref{lem:ch2_index2} imply
\[
\ch_3^{\overline{\beta}}(E) \leq - \frac{\overline{\beta}^2}{2} H^2 \cdot \ch_1^{\overline{\beta}}(E) - \overline{\beta}^3 \ch_0^{\overline{\beta}}(E) \leq 0.
\]

\subsection{Exceptional case of index two: the blow-up of $\P^3$ in a point}\label{subs:blow-up-point}
Let $X$ be the blow-up of $\P^3$ in a point. In this
case, $X$ has index $2$, and $H^3=7$. Denoting the exceptional divisor of
the blow-up by $e$ and  the pull-back of the hyperplane section of $\P^3$ to $X$ by $h$, we have
$H=2h-e$. Moreover, $e^3=h^3=1$, and $h\cdot e=0$ as a cycle in $A_1(X)$.
Finally, a calculation of $\ch_2(T_X)$ yields
\[
\ch_2(T_X) = 2h^2 + 2 e^2.
\]

For any $0 < \beta < \tfrac{1}{\sqrt{7}}$ and $C \geq 0$ let
\begin{align*}
f_C(\beta) &:= 7\left(\frac{\beta^2}{2} + C\right)\frac{1-\beta}{2} + \frac{\beta}{6}(7\beta^2 - 1), \\
g_C(\beta) &:= \left(\frac{\beta^2}{2} + C\right) + \frac{\beta}{6}(7\beta^2 - 1).
\end{align*}
We define $C_0$ to be the minimal real number such that $g_{C_0}(\beta)\geq 0$ for all $0 < \beta < \tfrac{1}{\sqrt{7}}$. Note that in this range we have $f_{C_0}(\beta) \geq g_{C_0}(\beta)$.

We will prove Theorem \ref{thm:main} in this case with
\[
\Gamma = \frac{h^2 + e^2}{6} + C_0 H^2.
\]

\begin{rmk}
{\rm It is not hard to compute
\[
C_0 = \frac{10\sqrt{30}}{1323} - \frac{3}{98}.
\]}
\end{rmk}

\begin{proof}[Proof of Theorem \ref{thm:main} for $X$]
Assume we have an object $E$ which is $\overline{\beta}$-stable with $0 \leq \overline{\beta}(E) < 1$ and $\ch_0(E) \geq 0$. Then Lemma \ref{lem:vanishing2} implies
\begin{align*}
\ch_3^{\overline{\beta}}(E) - \Gamma \cdot \ch_1^{\overline{\beta}}(E) &\leq - \left(\frac{\overline{\beta}^2}{2} + C_0\right) H^2 \cdot \ch_1^{\overline{\beta}}(E) - \frac{\overline{\beta}}{6}(7\overline{\beta}^2 - 1) \ch_0^{\overline{\beta}}(E) \\
&= -\left( 7\left(\frac{\overline{\beta}^2}{2} + C_0\right) \sqrt{\frac{\overline{\Delta}_H(E)}{\left(H^3 \cdot \ch_0(E)\right)^2}} + \frac{\overline{\beta}}{6}(7\overline{\beta}^2 - 1)\right) \ch_0^{\overline{\beta}}(E),
\end{align*}
where the second expression is only valid for $\ch_0(E) \neq 0$ and follows from the definition of $\overline{\beta}$. Since both $H^2 \cdot \ch_1^{\overline{\beta}}(E) \geq 0$ and $\ch_0^{\overline{\beta}}(E) \geq 0$, we are done if $\overline{\beta} = 0$, $\overline{\beta} \geq \tfrac{1}{\sqrt{7}}$, or $\ch_0^{\overline{\beta}}(E) = 0$. Assume $0 < \overline{\beta} < \tfrac{1}{\sqrt{7}}$ and $\ch_0^{\overline{\beta}}(E) \neq 0$.
Then if $\beta_+(E) \geq 1$, we get
\[
\ch_3^{\overline{\beta}}(E) - \Gamma \cdot \ch_1^{\overline{\beta}}(E) \leq -f_{C_0}(\overline{\beta})\ch_0^{\overline{\beta}}(E) \leq 0.
\]
If $\beta_+(E) < 1$, we can use Theorem \ref{thm:bogomolov_extension} to obtain
\[
\ch_3^{\overline{\beta}}(E) - \Gamma \cdot \ch_1^{\overline{\beta}}(E) \leq -g_{C_0}(\overline{\beta})\ch_0^{\overline{\beta}}(E) \leq 0. \qedhere
\]
\vspace{-1.17cm}

\hfill $\Box$

\vspace{0.1cm}
\end{proof}

\subsection{The index one case}

In this section we prove Theorem \ref{thm:main} when $X$ is a Fano threefold of index $1$. Let $H=-K_X$ and consider tilt stability with respect to it. The approach is exactly as in the index two case, but we have to distinguish two cases according to whether $0<\overline{\beta}(E)<1$ or $\overline{\beta}(E)=0$.

We have to define a $1$-cycle $\Gamma$. For $0 \leq \beta \leq 1$ and $C \in \R$ we define
\begin{align*}
f_{X,C}(\beta) &:= \frac{\beta^2}{2} - \frac{\beta}{2} + C, \\
g_X(\beta) &:= \frac{\beta^3}{6} - \frac{\beta^2}{4} + \beta \left( \frac{1}{12} + \frac{2}{H^3} \right) - \frac{1}{H^3}, \\
h_{X,C}(\beta) &:= f_{X,C}(\beta) \frac{1-\beta}{2} + g_X(\beta), \\
l_{X,C}(\beta) &:= \frac{f_{X,C}(\beta)}{H^3} + g_X(\beta).
\end{align*}

We have $g_X(0) = -\frac{1}{H^3} < 0$ and $g_X(1) = \frac{1}{H^3} > 0$. Therefore, we can define $\beta_0$ to be the largest zero of $g_X(\beta)$ in the interval $[0,1]$. Hence, $g_X(\beta)\geq0$, for all $\beta\in [\beta_0,1]$.

\begin{rmk}
{\rm If $H^3 \leq 48$, one gets $\beta_0 = \tfrac{1}{2}$. We refer to \cite[Table 12.3-6]{IP99:fano_varieties} for the fact that there are only seven types of Fano threefolds of degree strictly greater than $48$:

\begin{enumerate}
\item the blow up of $\P^3$ in a point, also given as $\P(\OO_{\P^2} \oplus \OO_{\P^2}(1))$, which has degree $d = 56$, Picard number $\rho = 2$, and index $i = 2$,
\item the product $\P^1 \times \P^2$, which has degree $d = 54$, Picard number $\rho=2$, and index $i = 1$,
\item the blow up of $\P^3$ along a line, also given as $\P(\OO_{\P^2}^{\oplus 2} \oplus \OO_{\P^2}(1))$, which has degree $d = 54$, Picard number $\rho = 2$, and index $i = 1$,
\item the projective bundle $\P(\OO_{\P^2} \oplus \OO_{\P^2}(2))$, which has degree $d = 62$, Picard number $\rho = 2$, and index $i = 1$,
\item the double blow up of $\P^3$ first in a point and then in a line contained in the exceptional divisor, which has degree $d = 50$, Picard number $\rho = 3$, and index $i = 1$,
\item the double blow up of $\P^3$ first in a point $x$ and then in the strict transform of a line through $x$, which has degree $d = 50$, Picard number $\rho = 3$, and index $i = 1$,
\item the projective bundle $\P(\OO_{\P^1 \times \P^1} \oplus \OO_{\P^1 \times \P^1}(1,1))$, which has degree $d = 52$, Picard number $\rho = 3$, and index $i = 1$.
\end{enumerate}}
\end{rmk}

Let $C_0 \geq 0$ be the minimal positive real number such that $C_0 H^3 \geq H \cdot \td_2(X) = \tfrac{H^3}{12} + 2$ and
\begin{equation*}
\begin{aligned}
  f_{X,C_0}(\beta) &\geq 0, && \text{ for all } \beta \in \left[0, 1\right] \\
  l_{X,C_0}(\beta) &\geq 0, && \text{ for all } \beta \in \left[0, \beta_0 \right].
\end{aligned}
\end{equation*}

By the classification of Fano threefolds, we have $H^3 \geq 4$. Therefore, we have $h_{X, C_0}(\beta) \geq l_{X, C_0}(\beta)$ in the interval $[0, \beta_0]$. We choose $\Gamma = C_0 H^2 - \td_2(X)$.

\begin{lem}
\label{lem:vanishing1}
Let $E$ be a $\overline{\beta}$-stable object with $0 < \overline{\beta}(E) < 1$.
Then
\[
\Hom(\OO_X(H),E) = \Ext^2(\OO_X(H),E) = 0.
\]
In particular,
\begin{align*}
0 &\geq \chi(\OO(H), E) = \ch_3(E) -\frac{1}{2}H \cdot \ch_2(E) + \td_2(X) \cdot \ch_1(E) - \ch_0(E) \\
& = \ch_3^{\overline{\beta}}(E) - \Gamma \cdot \ch_1^{\overline{\beta}}(E) + f_{X, C_0} (\overline{\beta}) H^2 \cdot \ch_1^{\overline{\beta}}(E) + g_X(\overline{\beta}) H^3 \cdot \ch_0^{\overline{\beta}}(E).
\end{align*}
\end{lem}

\begin{proof}
This proof is the same as for Lemma \ref{lem:vanishing2}.
\hfill $\Box$
\end{proof}

We are now in position to prove Theorem \ref{thm:main} when $E$ is a $\overline{\beta}$-stable object with $0 < \overline{\beta}(E) < 1$ and $\ch_0(E) \geq 0$.
In this situation Lemma \ref{lem:vanishing1} implies
\begin{align*}
\ch_3^{\overline{\beta}}(E) - \Gamma \cdot \ch_1^{\overline{\beta}}(E) &\leq - f_{X,C_0}(\overline{\beta}) H^2 \cdot \ch_1^{\overline{\beta}}(E) - g_{X}(\overline{\beta}) H^3 \cdot \ch_0^{\overline{\beta}}(E) \\
&= - \left(f_{X,C_0}(\overline{\beta}) \sqrt{\frac{\overline{\Delta}_H(E)}{\left(H^3 \cdot \ch_0(E)\right)^2}} + g_{X}(\overline{\beta}) \right) H^3 \cdot \ch_0^{\overline{\beta}}(E),
\end{align*}
where the second expression is only valid for $\ch_0(E) \neq 0$ and follows from the definition of $\overline{\beta}$.
As in the index two case, we will show that the right hand side is non-positive. Since both $H^2 \cdot \ch_1^{\overline{\beta}}(E) \geq 0$ and $\ch_0^{\overline{\beta}}(E) \geq 0$ we are done if $\overline{\beta} \geq \beta_0$, or $\ch_0^{\overline{\beta}}(E) = 0$.
If $0 < \overline{\beta} < \beta_0$ and $\ch_0^{\overline{\beta}}(E) \neq 0$ we will split the two cases $\beta_+(E) \geq 1$ and $\beta_+(E) < 1$.

If $\beta_+(E) \geq 1$, we get
\[
\ch_3^{\overline{\beta}}(E) - \Gamma \cdot \ch_1^{\overline{\beta}}(E) \leq -h_{X,C_0}(\overline{\beta}) H^3 \cdot \ch_0^{\overline{\beta}}(E) \leq 0.
\]
If $\beta_+(E) < 1$, we can use Theorem \ref{thm:bogomolov_extension} to obtain
\[
\ch_3^{\overline{\beta}}(E) - \Gamma \cdot \ch_1^{\overline{\beta}}(E) \leq -l_{X,C_0}(\overline{\beta}) H^3 \cdot \ch_0^{\overline{\beta}}(E) \leq 0.
\]

We are left to deal with the case of a $\overline{\beta}$-stable object $E$ with $\ch_0(E) \geq 0$ and $\overline{\beta}(E) = 0$.
If $\overline{\Delta}_H(E)=0$ then, by Theorem \ref{thm:bogomolov_extension}, $E$ is a shift of $\OO_X$ or an ideal of points. In those cases, Theorem \ref{thm:main} is obvious.

\begin{lem}
\label{lem:ext1_vanishing}
Assume that there is a $\overline{\beta}$-stable object $E$ satisfying $\ch_0(E) \geq 0$, $\overline{\beta}(E) = 0$, and $\ch_3(E) > {\Gamma \cdot \ch_1(E)}$, i.e., $E$ contradicts Theorem \ref{thm:main}. Then $E$ can be chosen such that additionally $\Ext^1(E,\OO_X)=0$.
\begin{proof}
Pick a $\overline{\beta}$-stable object $E$ such that $H^2 \cdot \ch_1(E)$ is minimal with the properties $\ch_0(E) \geq 0$, $\overline{\beta}(E) = 0$, and $\ch_3(E) - \Gamma \cdot \ch_1(E) > 0$. Then $\overline{\Delta}_H(E)\neq0$, and equivalently $H^2 \cdot \ch_1(E)\neq0$. If $E$ could be chosen with $\ch_0(E)$ arbitrarily large, the conditions of Theorem \ref{thm:bogomolov_extension} would be fulfilled. This would imply
\[
\left(\frac{H^2 \cdot \ch_1(E)}{H^3 \cdot \ch_0(E)}\right)^2 = \frac{\overline{\Delta}_H(E)}{(H^3 \cdot \ch_0(E))^2}  \geq \frac{1}{(H^3)^2}
\]
even in the limit $\ch_0(E) \to \infty$, a contradiction. Therefore, we can additionally choose $E$ such that $\ch_0(E)$ is maximal for objects with these properties.

If $\Ext^1(E, \OO_X) \neq 0$, we can get a non trivial triangle
\[
E' \to E \to \OO_X[1] \to E'[1].
\]
We want to show that $E'$ is in $\Coh^0(X)$, i.e., that the morphism $E \to \OO_X[1]$ is surjective. Let $T$ be the cokernel of this map in $\Coh^0(X)$. Since $\OO_X[1]$ is semistable, we get $\nu_{0, \alpha}(T) = \infty$. Since it is even stable, this means $T$ a sheaf supported on points (this is because those are the only objects that are mapped to the origin by $Z_{0, \alpha}$). But then $\Hom(\OO_X[1], T) = 0$, a contradiction unless $T = 0$.

We will show that $E'$ is also $\overline{\beta}$-stable to get a contradiction to the maximality of $\ch_0(E)$. Let
\[
0 = E_0 \into E_1 \into \ldots \into E_n = E'
\]
be the Harder-Narasimhan filtration of $E'$ for $\beta = 0$ and $\alpha \ll 1$. The fact that $E$ is $\overline{\beta}$-stable and that $E' \to E$ is injective in $\Coh^0(X)$ implies
\[
\nu_{0, \alpha}(E') \leq \nu_{0, \alpha}(E_i) \leq \nu_{0, \alpha}(E).
\]
Taking the limit $\alpha \to 0$ implies $H \cdot \ch_2(E_i) = 0$ for all $i$. For every semistable factor $E_i/E_{i-1}$ we choose a Jordan H\"older filtration and call the stable factors of all those filtrations $F_1, \ldots, F_m$. Then $H \cdot \ch_2(F_i) = 0$ for all $i = 1, \ldots, m$. The values $\ch_3(F_i) - \Gamma \cdot \ch_1(F_i)$ add up to $\ch_3(E') -  \Gamma \cdot \ch_1(E') = \ch_3(E) - \Gamma \cdot \ch_1(E) > 0$. Choose $j$ such that $\ch_3(F_j) > \Gamma \cdot \ch_1(F_j)$. By definition of $\Coh^0(X)$ we have $H^2 \cdot \ch_1(F_i) \leq  H^2 \cdot \ch_1(E') = H^2 \cdot \ch_1(E)$ for all $i = 1, \ldots, m$.

Assume $H^2 \cdot \ch_1(F_j) < H^2 \cdot \ch_1(E)$.
If $\ch_0(F_j)\geq0$, then this contradicts the minimality of $E$.
If $\ch_0(F_j)<0$, then we can use the derived dual via Proposition \ref{prop:tilt_derived_dual} to reduce to the positive rank case.
Hence, we must have $H^2 \cdot \ch_1(F_j) = H^2 \cdot \ch_1(E)$ and $H^2 \cdot \ch_1(F_i) = 0$ for $i \neq j$. Since $H^2 \cdot \ch_1(E)\neq0$, the existence of the morphism $F_1 \into E$ shows that $j = 1$. But the slopes of the factors in the Harder-Narasimhan filtration are strictly decreasing. Therefore, we must have $m = 1$ and $E'$ is indeed $\overline{\beta}$-stable.
\hfill $\Box$
\end{proof}
\end{lem}

We can now finish the proof of Theorem \ref{thm:main} in the case of a $\overline{\beta}$-stable object $E$ with $\ch_0(E) \geq 0$ and $\overline{\beta}(E) = 0$. Assume there is $E$ as in Lemma \ref{lem:ext1_vanishing} contradicting the theorem, i.e.,
\[
\ext^2(\OO(H), E) = \ext^1(E, \OO) = 0.
\]
By stability we have $\hom(\OO(H), E) = 0$ and we get $\chi(\OO(H), E) \leq 0$. From here, the proof is finished with Theorem \ref{thm:bogomolov_extension} as before.

\section{Toric threefolds}\label{sec:examples1}

In this section, we use a variant of the method in \cite{BMS14:abelian_threefolds} to prove Conjecture \ref{conj:bmt} in some toric (not necessarily Fano) cases, with respect to certain polarizations. The results presented here arose from discussions together with Arend Bayer and Paolo Stellari.

\begin{thm}\label{thm:toric2}
Let $X$ be a smooth projective complex toric threefold.
Let $H$ be an ample divisor such that, for all effective divisors $D$ on $X$, we have $H \cdot D^2\geq 0$, and $H \cdot D^2 = 0$ implies
that $D$ is an extremal ray of the effective cone.
Then, for any $\nu_{\alpha,\beta}$-stable object $E\in\Coh^\beta(X)$ with $\nu_{\alpha,\beta}(E)=0$, we have
\[
\ch_3^\beta(E) - \frac{\alpha^2}{6} H^2 \cdot \ch_1^\beta(E) \leq 0.
\]
\end{thm}

Since the statement of Theorem \ref{thm:toric2} is independent of scaling $H$, we will assume throughout this section that $H$ is primitive.
The extra condition on the extremality of divisors with $H \cdot D^2=0$ is probably not necessary.
It is trivially satisfied by $\P^2$-bundles over $\P^1$.
For us, it will simplify the proof, since it directly implies (see e.g., \cite[Lemma 15.1.8]{CLS11:toric_varieties}) that a primitive such $D$ is the class of an irreducible torus-invariant divisor.
Also, since $X$ is a threefold, there cannot be more than $3$ such irreducible torus-invariant divisors with the same class.
Indeed, if $D$ is not movable, then there is only one irreducible divisor with that class. If $D$ is movable, then since it is extremal in the effective cone, it cannot be big and so it induces a rational morphism on a lower dimensional toric variety of Picard rank $1$, namely $\P^1$ or $\P^2$.

The positivity condition $H \cdot D^2\geq 0$ is instead necessary in the proof.
As proved in \cite[Corollary 3.11]{BMS14:abelian_threefolds}, this implies that all (shifts of) line bundles in $X$ are $\nu_{\alpha,\beta}$-stable, for all $\alpha,\beta\in\R$, $\alpha>0$.

The assumptions are satisfied for any polarization on $\P^1\times\P^2$ and on $\P^1\times\P^1\times\P^1$.
In the blow-up of $\P^3$ in a line, if we denote the pull-back of $\OO_{\P^3}(1)$ by $h$ and the pull-back of $\OO_{\P^1}(1)$ by $f$, then the assumption is satisfied by any polarization of the form $ah + bf$, with $a, b > 0$ and $a \leq b$. The class of the anticanonical bundle is $3h + f$, and so it is not covered by the above result.

In order to prove Theorem \ref{thm:toric2}, we use
the toric Frobenius and the formula from Theorem \ref{thm:ToricDecomposition} below as follows.
First, as in the previous sections, we can work only with $\overline{\beta}$-stable objects.
In the case where $\overline{\beta}(E)=0$, assuming $\ch_3(E) >0$, the Euler characteristic
$\chi(\OO_X,\underline{m}^*E)$ grows as a polynomial in $m$ of degree $3$. On the other hand, by using adjointness, Theorem \ref{thm:ToricDecomposition},
and stability of line bundles, we can bound both $\hom(\OO_X,
\underline{m}^*E)$ and $\ext^2(\OO_X,\underline{m}^*E)$ and show
that $\chi(\OO_X,\underline{m}^*E)$ has to grow at most as a polynomial in $m^2$, thus giving a contradiction to $\ch_3(E)>0$.
The basic idea is then to reduce to this case.
When $\overline{\beta}(E)$ is an integer, this is easy to do by tensoring with line bundles. When it is a rational number, we pull-back again via the toric Frobenius morphism to simplify denominators.
Finally, the irrational case can be dealt with by elementary Dirichlet approximation.

\subsection{Toric Frobenius morphism}\label{subsec:Frobenius}

Let $X$ be a smooth projective toric threefold.
We denote the irreducible torus-invariant divisors by $D_\rho$ (corresponding to the rays $\rho$ of the fan associated to $X$).
The canonical line bundle is then $\omega_X = \OO_X\left(- \sum_\rho D_\rho\right)$.

For an integer $m\in\N$, we denote the \emph{toric Frobenius morphism} by
\[
\underline{m}\colon X \to X.
\]
It is defined via multiplication by $m$ on the lattice and is a finite flat map of degree $m^3$. The main property we will need is the following result from {\cite{Thom00:frobenius} on direct images of line bundles (see also \cite{Ach15:Frobenius} for a short proof).

\begin{thm}\label{thm:ToricDecomposition}
Let $D\in\Pic(X)$ be a line bundle on $X$.
Then
\[
\underline{m}_* D = \bigoplus_j \left(L_j^\vee\right)^{\oplus \eta_j},
\]
where
\begin{itemize}
\item the line bundles $L_j$ are given by all the possible integral divisors in the formula
\[
\OO_X\left(\frac 1m \left( - D + \sum_\rho a_\rho D_\rho \right)\right),
\]
where $a_\rho$ varies in between $0, \ldots, m-1$,

\item the multiplicity $\eta_j$ counts the number of $a_\rho$'s giving the same line bundle $L_j$.
\end{itemize}
\end{thm}

\begin{ex}{\rm
Let $X=\P^2\times\P^1$ and let us denote the pull-backs of the hyperplane classes from $\P^2$ and $\P^1$ by $h$ and $f$ respectively, as before. There are five torus-invariant  irreducible divisors, three of them lie in the class $h$ and two of them in the class $f$.
It follows that
$$\underline{m}_* \OO =\OO \oplus \OO(-f)^{\oplus (m-1)} \oplus \OO(-h)^{\oplus r_1}
\oplus \OO(-h-f)^{\oplus r_2} \oplus \OO(-2h)^{\oplus r_3} \oplus \OO(-2h-f)^{\oplus r_4},
$$
where $r_1, r_3$ grow like $m^2$ and $r_2, r_4$ grow like $m^3$. }
\end{ex}

\begin{ex}\label{ex:NonStablePullBack}{\rm
Notice that Frobenius pull-back in general does not preserve $\overline{\beta}$-stability. Indeed, in the case of the blow-up of $\P^3$ in a point, we can consider $\OO(h)$. It is not too hard to check that $\OO(h)$ is $\overline{\beta}$-stable, while for any $m\geq 2$, the pull-back
$\underline{m}^* \OO(h)= \OO(mh)$ is not $\overline{\beta}$-semistable (see \cite{Sch16:counterexample} and Section \ref{sec:examples2} for more details).}
\end{ex}

\subsection{Proof of Theorem \ref{thm:toric2}: the integral case}\label{subsec:integral}

We use the statement in \cite[Conjecture 5.3]{BMS14:abelian_threefolds} (or Lemma \ref{lem:reduction_alpha_0} in the present note).
Let $E$ be a $\overline{\beta}$-stable object.
In this section, we assume that $\overline{\beta}(E)\in\Z$.
By tensoring by multiples of $\OO_X(H)$, we can assume that $\overline{\beta}(E)=0$.
We want to show $\ch_3(E)\leq 0$.

Assume the contrary $\ch_3(E)>0$.
By using the Hirzebruch-Riemann-Roch Theorem we can compute
\[
\chi(\OO_X,\underline{m}^*E) = m^3 \ch_3(E) + O(m^2).
\]
If $\underline{m}^*E$ were $\overline{\beta}$-semistable, we could easily get a contradiction.
Unfortunately, in general, it is not true, as remarked in Example \ref{ex:NonStablePullBack}.
But we can still use the push-forward $\underline{m}_*$ and Theorem \ref{thm:ToricDecomposition}.
Indeed, by adjointness, we have
\[
\chi(\OO_X,\underline{m}^*E) = \chi \left((\underline{m}_*\OO_X)^\vee, E\right) \leq \hom \left((\underline{m}_*\OO_X)^\vee, E\right) + \ext^2 \left((\underline{m}_*\OO_X)^\vee, E\right).
\]
The last inequality follows since, given a line bundle $L$ on $X$, we have $\hom(L,E[i])=0$ if $i< -1$ or $i>3$.
We want to prove that the right hand side has order $\leq m^2$.
In the notation of Theorem \ref{thm:ToricDecomposition}, we have $\underline{m}_*\OO_X = \oplus (L_j^\vee)^{\oplus \eta_j}$, where $L_j = \OO_X(\frac{1}{m} \sum_\rho a_\rho D_\rho)$.

First of all, since $H$ is ample and $L_j$ is effective, we have
\[
H^2 \cdot L_j \geq 0
\]
and $H^2 \cdot L_j=0$ if and only if $L_j=\OO_X$.
Hence, $L_j\in\Coh^{\beta=0}(X)$, when $L_j\neq\OO_X$, and as remarked before, by \cite{BMS14:abelian_threefolds}, they are all $\nu_{\alpha,0}$-stable, for all $\alpha>0$.
If $H \cdot L_j^2>0$, then
\[
\lim_{\alpha\to 0}\nu_{\alpha,0}(L_j) > 0 = \lim_{\alpha\to 0} \nu_{\alpha,0}(E)
\]
and so $\Hom(L_j,E)=0$.
If $H \cdot L_j^2=0$, then by our assumption $a_\rho=0$ for all but at most $3$ rays $\rho$. Hence, $\eta_j$ has order at most $m^2$: there are at most three non-trivial coefficients $a_\rho$, with the property $0 \leq a_\rho < m$, and related by a linear equation defining $L_j$ in Theorem \ref{thm:ToricDecomposition}.
Summing up, we have
\[
\hom\left((\underline{m}_*\OO_X)^\vee,E \right) = \hom(\OO_X,E) + \sum_{j\,:\, H \cdot L_j^2=0} \eta_j \cdot \hom(L_j,E) = O(m^2).
\]

To bound the $\ext^2$, by Serre duality, we have
\[
\ext^2(L_j,E) = \hom(E, L_j\otimes \omega_X [1]).
\]
Since $\omega_X = \OO_X(-\sum_\rho D_\rho)$, we have
\[
L_j' := L_j\otimes \omega_X = \OO_X\left(-\frac{1}{m} \sum_\rho (m-a_\rho) D_\rho \right).
\]
Since $0\leq a_\rho < m$, then $m - a_\rho >0$, for all $\rho$.
Therefore, $H^2 \cdot L_j'<0$ and $H \cdot (L_j')^2>0$.
Hence, $L_j'[1]\in\Coh^{\beta=0}(X)$ and
\[
\lim_{\alpha\to 0}\nu_{\alpha,0}(L_j'[1]) < 0 = \lim_{\alpha\to 0} \nu_{\alpha,0}(E).
\]
Again by stability of $L_j'[1]$ we have
\[
\hom(E, L_j\otimes \omega_X [1])=0
\]
for all $j$.
It follows that
\[
\chi(\OO_X,\underline{m}^*E) = O(m^2),
\]
giving the required contradiction.

\subsection{Proof of Theorem \ref{thm:toric2}: the rational case}\label{subsec:rational}

In this section, we assume that $\overline{\beta}(E)\in\Q\setminus\Z$.
We write $\overline{\beta}(E)=\frac{p}{q}$, with $p$ and $q$ coprime, $q>0$.
We want to show that $\ch_3^{p/q}(E)\leq0$.
As before, by the Hirzebruch-Riemann-Roch Theorem, we have
\[
\chi\left(\OO_X,\underline{m}^*\left( \underline{q}^*E \otimes \OO_X(-pH) \right) \right) =
m^3 q^3 \ch_3^{p/q}(E) + O(m^2),
\]
and, by adjointness,
\begin{align*}
&\chi\left(\OO_X,\underline{m}^*\left( \underline{q}^*E  \otimes \OO_X(-pH) \right) \right) =
\chi\left((\underline{mq}_* \OO_X(-mpH))^\vee,E\right)\\
& \qquad \leq \hom\left((\underline{mq}_* \OO_X(-mpH))^\vee,E\right) +
\ext^2\left((\underline{mq}_* \OO_X(-mpH))^\vee,E\right).
\end{align*}
By Theorem \ref{thm:ToricDecomposition},
\[
\underline{mq}_* \OO_X (-mpH) = \bigoplus_j \left(L_j^\vee\right)^{\oplus \eta_j},
\]
where
\[
L_j = \OO_X\left( \frac{1}{mq} \left( mpH + \sum_\rho a_\rho D_\rho\right)\right),
\]
and $0 \leq a_\rho < mq$.
Therefore,
\[
\ch^{p/q}_1(L_j)= L_j \otimes \OO_X\left(- \frac{p}{q}H\right) = \OO_X \left(\frac{1}{mq} \sum_\rho a_\rho D_\rho \right)
\]
is an effective divisor and cannot be $\OO_X$ because $p/q$ is not an integer.
It follows that $H^2 \cdot \ch_1^{p/q}(L_j)>0$ for all $j$.
Moreover
\[
H \cdot \ch_2^{p/q}(L_j)= \frac{1}{2} H \cdot \ch_1^{p/q}(L_j)^2\geq 0.
\]
As in the $\overline{\beta}(E) = 0$ case, the equality holds if and only if the corresponding $\eta_j$ has order at most $m^2$, since $q$ is constant.
If $H \cdot \ch_2^{p/q}(L_j) > 0$, then
\[
\lim_{\alpha\to 0}\nu_{\alpha,p/q}(L_j) > 0 = \lim_{\alpha\to 0} \nu_{\alpha,p/q}(E)
\]
and so $\hom(L_j,E)=0$.
As in the integral case, this shows that
\[
\hom\left((\underline{mq}_* \OO_X(-mpH))^\vee, E \right) = O(m^2).
\]
The vanishing of $\ext^2$ follows as in the integral case, by using Serre duality and stability.

\subsection{Proof of Theorem \ref{thm:toric2}: the irrational case}\label{subsec:irrational}

Finally, we assume that $\overline{\beta}(E)\in\R\setminus\Q$.
By assumption, there exists $\epsilon>0$ such that $E$ is  $\nu_{\alpha,\beta}$-stable for all $(\alpha,\beta)$ in
\[
V_\epsilon := \left\{(\alpha,\beta)\in\R_{>0}\times\R\,:\, 0< \alpha<\epsilon,\,  \overline{\beta}(E)-\epsilon<\beta <\overline{\beta}(E)+\epsilon \right\}.
\]

By the Dirichlet approximation theorem, there exists a sequence $\left\{\beta_n=\frac{p_n}{q_n}\right\}_{n\in\N}$ of rational numbers such that
\begin{equation*}\label{eq:Dirichlet}
\left| \overline{\beta}(E) - \frac{p_n}{q_n} \right| < \frac{1}{q_n^2}< \epsilon
\end{equation*}
for all $n$, where $q_n \to +\infty$ as $n\to+\infty$.

We compute, for $n\gg0$,
\[
\chi\left(\OO_X,\underline{q_n}^*E\otimes \OO_X(-p_nH)\right) = q_n^3 \ch_3^{p_n/q_n}(E) + O(q_n^2) \geq q_n^3 \ch_3^{\overline{\beta}(E)}(E) + O(q_n^2).
\]
The last inequality follows since, by definition, $\ch_3^\beta(E)$ has a local minimum at $\beta=\overline{\beta}(E)$.

As in the previous cases, we assume for a contradiction that $\ch_3^{\overline{\beta}(E)}(E)>0$, and we want to bound
\[
\chi\left((\underline{q_n}_* \OO_X(-p_nH))^\vee,E\right)
\leq \hom\left((\underline{q_n}_* \OO_X(-p_nH))^\vee,E\right) +
\ext^2\left((\underline{q_n}_* \OO_X(-p_nH))^\vee,E\right)
\]
for $n\gg0$.
By Theorem \ref{thm:ToricDecomposition},
\[
\underline{q_n}_* \OO_X (-p_n H) = \bigoplus_j \left(\left(L_j^{(n)}\right)^\vee\right)^{\oplus \eta_j^{(n)}},
\]
where
\[
L_j^{(n)} =\OO_X\left( \frac{1}{q_n} \left( p_nH + \sum_\rho a_\rho^{(n)} D_\rho\right)\right),
\]
and $0 \leq a_\rho^{(n)} < q_n$.
Notice that, since $L_j^{(n)}$ is an integral divisor, $\frac{p_n}{q_n}$ and  $\frac{a_\rho^{(n)}}{q_n}$ are both universally bounded with respect to $n$, there is only a finite number of isomorphism classes of $L_j^{(n)}$ for all $n$.

We have that
\[
\ch^{p_n/q_n}_1(L_j^{(n)})= L_j^{(n)} \otimes \OO_X\left(- \frac{p_n}{q_n}H\right) = \OO_X \left(\frac{1}{q_n} \sum_\rho a_\rho^{(n)} D_\rho \right)
\]
is an effective divisor, and cannot be $\OO_X$ for $n \gg 0$.
Since $\sum_\rho a_\rho^{(n)} D_\rho$ is an integral divisor,
we have
\[
H^2 \cdot \sum_\rho a_\rho^{(n)} D_\rho \geq 1,
\]
so that
\[
H^2 \cdot \ch_1^{p_n/q_n}\left(L_j^{(n)}\right) \geq \frac{1}{q_n}.
\]
Now,
\begin{equation*}\label{eq:c1bar}
H^2 \cdot \ch_1^{\overline{\beta}}\left(L_j^{(n)}\right) = H^2 \cdot \ch_1^{p_n/q_n}\left(L_j^{(n)}\right) + H^3 \left(\frac{p_n}{q_n}-\overline{\beta}\right) \geq \frac{1}{q_n} -  \frac{H^3}{q_n^2} > 0
\end{equation*}
for $n \gg 0$. Therefore $L_j^{(n)}$ belongs to $\Coh^{\overline{\beta}}(X)$ for $n \gg 0$.
By definition,
\begin{align*}
H \cdot \ch_1^{\overline{\beta}}\left(L_j^{(n)}\right)^2 &= H \cdot \left( \ch_1^{p_n/q_n}\left(L_j^{(n)}\right) + \left(\frac{p_n}{q_n}-\overline{\beta}\right)H \right)^2\\
&> H \cdot \left( \ch_1^{p_n/q_n}\left(L_j^{(n)}\right)\right)^2 + 2 \left( \frac{p_n}{q_n}-\overline{\beta} \right) H \cdot \ch_1^{p_n/q_n}\left(L_j^{(n)}\right)\\
&> \frac{1}{q_n^3} \, H \cdot \left( q_n \left( \sum a_\rho^{(n)} D_\rho\right)^2 - 2 H \cdot \sum a_\rho^{(n)} D_\rho \right). \end{align*}

If
\[
f_{n,j} := H \cdot \ch_2^{\overline{\beta}}\left(L_j^{(n)}\right) = \frac{1}{2} H \cdot \ch_1^{\overline{\beta}}\left(L_j^{(n)}\right)^2 > 0,
\]
we can argue as in the previous cases to show that
$\Hom(L_j^{(n)},E)=0$ by stability. Since $q_n \to \infty$, we
can have $f_{n,j} \leq 0$ only if the ratio
\[
\frac{2 H^2 \cdot \sum a_\rho^{(n)} D_\rho}{H \cdot \left( \sum a_\rho^{(n)} D_\rho\right)^2}
\]
is not bounded from above for $n \gg 0$, since $H \cdot  \left( \sum a_\rho^{(n)} D_\rho\right)^2 \geq 0$ by assumption.
But
\[
2 H^2 \cdot \sum a_\rho^{(n)} D_\rho \leq K \cdot \sum a_\rho^{(n)},
\]
for a constant $K>0$ which is independent on $n$ and $j$, and
\[
H \cdot \left( \sum a_\rho^{(n)} D_\rho\right)^2 = \sum_{\rho,\tau} a_\rho^{(n)} a_\tau^{(n)} (H \cdot D_\rho \cdot D_\tau).
\]
Fix $\rho_0$ such that $a_{\rho_0}^{(n)} \neq 0$.
Then
\[
a_{\rho_0}^{(n)} \leq \sum_{\tau} a_{\rho_0}^{(n)} a_\tau^{(n)} (H \cdot D_{\rho_0} \cdot D_\tau)
\]
unless $a_\tau^{(n)}=0$ for all $\tau$ for which $H \cdot D_{\rho_0} \cdot D_\tau\neq 0$. That is, we can have $a_\tau^{(n)} \neq 0$ possibly only for $\tau$ such that $H \cdot D_{\rho_0} \cdot D_\tau = 0$.
By our assumption, the latter equality means that $D_\tau$ has the same divisor class as $D_{\rho_0}$, $H \cdot D_{\rho_0}^2=0$.
Indeed, we can show first that $H \cdot D_{\rho_0} \cdot D_\tau = 0$ implies that $H \cdot D_\tau^2 \leq 0$. Assume not and set $\lambda = \tfrac{H^2 \cdot D_{\rho_0}}{H^2 \cdot D_\tau}$. The Hodge Index Theorem says
\[
0 \geq H \cdot (D_{\rho_0} - \lambda D_\tau)^2 = \lambda^2 H \cdot D_\tau^2 > 0,
\]
a contradiction.

By the assumption in the theorem this implies $H \cdot D_\tau^2 = 0$. Thus, we get $H \cdot (D_\tau + D_{\rho_0})^2 = 0$. By the extremality assumption, we then have that $D_\tau$ is in the same class as $D_{\rho_0}$.
It follows that $H \cdot \left(\ch_1^{p_n/q_n}\left(L_j^{(n)}\right)\right)^2=0$.

Moreover, as in the integral case, we can bound the number of torus-invariant divisors in the same class of such a $D_{\rho_0}$. Summing up, if $f_{n,j}\leq0$, then the multiplicity $\eta_j^{(n)}$ has order at most $q_n^2$.
Together with the finiteness of the isomorphism classes of
$L_j^{(n)}$, we have that
\[
\hom\left((\underline{q_n}_* \OO_X(-p_nH))^\vee,E\right) = \sum_{j\,:\,H.\left(L_j^{(n)}\right)^2=0} \eta_j^{(n)} \hom(L_j^{(n)},E) = O(q_n^2).
\]

As before, the vanishing of $\ext^2$ follows as in the integral case, by using Serre duality and stability.
This shows that $\ch_3^{\overline{\beta}(E)}(E)\leq0$ also in this case, and therefore completes the proof of Theorem \ref{thm:toric2}.

\section{Details about the blow-up of $\P^3$ in a point}
\label{sec:examples2}

Theorem \ref{thm:main} implies the existence of Bridgeland stability conditions on all Fano threefolds. However, it would be interesting to know what the optimal class $\Gamma$ is. A condition that is coherent with the case of Picard rank $1$ would be $\Gamma \cdot H = 0$.

In this section we will study the blow-up of $\P^3$ in a point more carefully. We use the notations for divisors on $X$ which were introduced in Section \ref{subs:blow-up-point}. In \cite{Sch16:counterexample}, it was shown that the line bundle $\OO(h)$ does not satisfy Theorem \ref{thm:main} with $\Gamma=0$.
In \cite{Mar16:counterexample}, it was shown that the structure sheaf $\OO_e$ of the exceptional divisor also does not satisfy Theorem \ref{thm:main} with $\Gamma=0$.
We will do the following computation.

\begin{prop}
Line bundles on the blow-up of $\P^3$ in a point and $\OO_e$ satisfy Theorem \ref{thm:main} with $\Gamma=k(h^2+2e^2)$, for $\frac{1}{48}\leq k\leq \frac{3}{98}+\frac{2\sqrt{2}}{147}$.
\end{prop}

Note that here $\Gamma \cdot H=0$.

\begin{proof}
Recall that $H=2h-e$, so by twisting a suitable choice of $\OO(H)$, we can assume that the line bundle is of the form $\OO(mh)$ for an integer $m$.

The hyperbola $\nu_{\alpha, \beta}(\OO(mh)) = 0$ is given by the equation
\[
7\alpha^2-7\beta^2+8m\beta-2m^2=0.
\]
So when $m \leq 0$, then $\overline\beta = \frac{4+\sqrt 2}{7}m$.
Evaluating the left hand side of the inequality in Theorem \ref{thm:main}, we get
\begin{align*}
&\;\;\;\;\;\ch_3^\beta(\OO(mh)) - \Gamma \cdot \ch_1^\beta(\OO(mh)) - \frac{\alpha^2}{6} H^2 \cdot \ch_1^{\beta}(\OO(mh)) \\
&= \frac{m^3}{6}-\overline \beta m^2+ 2\overline \beta ^2m-\frac{7}{6}\overline \beta ^3-km \\
&= \left(\frac{3}{98}+\frac{2\sqrt{2}}{147}\right)m^3-km.
\end{align*}
So when $k\leq \frac{3}{98}+\frac{2\sqrt{2}}{147}$, the above function is less than or equal to $0$ for any integer $m\leq 0$, and the theorem holds.

When $m=1$ holds, then $\overline\beta = \frac{4-\sqrt 2}{7}$. A similar computation shows that the inequality holds for $m=1$ when $k\geq \frac{3}{98}-\frac{2\sqrt{2}}{147}$.

Now we focus on the case when $m\geq 2$. In this case $\OO(mh)$ is not $\overline\beta$-stable. To see this, just note that there always exists a nontrivial map $\OO(mh-e) \to \OO(mh)$. The wall induced by this map is given by the equation
\[
7\alpha^2+7\beta^2-7\beta-2m^2+4m=0.
\]
When $m\geq 2$, this wall is non-empty, and $\OO(mh)$ is destabilized on the wall. So in order to show the theorem, it suffices to check the inequality at this wall.

Note that the top point of the wall, which is also the intersection with the hyperbola $\nu_{\alpha, \beta}(\OO(mh)) = 0$, has coordinates $\beta=\frac{1}{2}$ and $\alpha^2=\frac{2m^2-4m+7/4}{7}$. So we have
\begin{align*}
&\;\;\;\;\;\ch_3^\beta(\OO(mh)) - \Gamma \cdot \ch_1^\beta(\OO(mh)) - \frac{\alpha^2}{6} H^2 \cdot \ch_1^{\beta}(\OO(mh)) \\
&= \frac{m^3}{6}-\beta m^2+ 2\beta ^2m-\frac{7}{6} \beta ^3-km - \frac{2m^2-4m+7/4}{42}(4m-7\beta)\\
&= -\frac{m^3}{42}+\frac{m^2}{21}-km.
\end{align*}
It is easy to see that this function is negative for $m\geq 2$ and $k>0$.

For the structure sheaf $\OO_e$, we have $\overline{\beta}=\frac{1}{2}$ and
\[
\ch_3^{1/2}(\OO_e) - \Gamma\cdot \ch_1^{1/2}(\OO_e) = \frac{1}{24} - 2k,
\]
from which the result follows.
\hfill $\Box$
\end{proof}



\bibliographystyle{amsalpha}
\bibliographymark{References}

\def\cprime{$'$} \def\cprime{$'$}

\end{document}